\documentclass{article}
\usepackage{fullpage}
\usepackage[utf8]{inputenc}
\usepackage{amsmath, amsfonts,amsthm, bm}
\usepackage{mathtools}
\usepackage{thm-restate}
\usepackage[dvipsnames]{xcolor}
\usepackage{caption}
\usepackage{hyperref}
\hypersetup{
    colorlinks=true,       
    linkcolor=blue,        
    citecolor=red,         
    filecolor=magenta,     
    urlcolor=cyan,         
    linktocpage=true
}
\usepackage[capitalise,noabbrev]{cleveref}

\usepackage{tikz}
\usetikzlibrary{arrows,arrows.meta, positioning,calc,shapes.geometric, math}
\usepackage{adjustbox}
\definecolor{greenish}{RGB}{27,158,119}
\definecolor{MyOrange}{RGB}{217,95,2}
\definecolor{MyPurple}{RGB}{117,112,179}

\usepackage[skip=5pt]{parskip}
\usepackage{setspace}
\usepackage[sort]{cite}

\usepackage{array}
\newcolumntype{C}[1]{>{\centering\arraybackslash}m{#1}}
\usepackage{colortbl,arydshln,booktabs,multirow}
\theoremstyle{plain}
\newtheorem{thm}{Theorem}
\newtheorem{lem}[thm]{Lemma}

\newtheorem{cor}[thm]{Corollary}

\theoremstyle{definition}
\newtheorem{defn}[thm]{Definition}

\newcommand{\p}[1]{p_{(#1)}}
\usepackage{authblk}

\title{New upper bounds on the order of mixed cages of girth 6}
\author[1]{Gabriela Araujo-Pardo}
\author[2]{Lydia Mirabel Mendoza-Cadena}
\affil[1]{Instituto de Matemáticas, Universidad Nacional Autónoma de México, Campus Juriquilla, Querétaro, Mexico.}
\affil[2]{Center for Mathematical Modeling (CNRS IRL2807), Universidad de Chile, Santiago, Chile}
\date{ }

\begin{document}

\maketitle


\begin{abstract}
    A $[z,r;g]$-mixed cage is a mixed graph of minimum order such that each vertex has $z$ in-arcs, $z$ out-arcs, $r$ edges, and it has girth $g$, and the minimum order for $[z,r;g]$-mixed graphs is denoted by  $n[z,r;g]$. 
    
    In this paper, we present an infinite family of mixed graphs with girth $6$, that improves, in some cases, the families that we give in G. Araujo-Pardo and L. Mendoza-Cadena. \textit{On Mixed Cages of girth 6}, arXiv:2401.14768v2.

    In particular, if $q$ is an even prime power we construct a family of graphs that satisfies $n[\frac{q}{4},q;6]\leq 4q^2-4$, and if $q$ is an odd prime power, and
    $\frac{q-3}{2}$ is odd then our family satisfies that $n[\frac{q-1}{4},q;6]\leq 4q^2-4$, otherwise 
    $n[\frac{q-3}{4},q;6]\leq 4q^2-4$.
        
 \medskip
    
\end{abstract}

In this manuscript, we consider simple graphs: no parallel arcs or edges are permitted, and neither can there be a parallel arc and edge.

Let $G=(V;E\cup A)$ be a mixed graph where $V(G)$ denotes the set of vertices, $E(G)$ the set of (undirected) edges, and $A(G)$ the set of arcs. To distinguish clearly between these types, we represent an edge between vertices $u$ and $v$ as $uv$, and a directed arc from $u$ to $v$ as $(u,v)$. Two vertices $u$ and $v$
are said to be edge-adjacent (respectively, arc-adjacent) if there exists an edge (respectively, an arc) between them. We represent walks, paths, and cycles as sequences of the form $(v_0, v_1, \dots, v_n)$, where each consecutive pair $v_i$ and $v_{i+1}$ is either edge- or arc-adjacent. The \emph{girth} of the graph is defined as the length of its shortest cycle. Given a subset $X \subseteq V$, the graph $G[X]$ is obtained by removing all vertices in $V \setminus X$ of $G$  along with any edges or arcs incident to them.

A mixed graph is called \emph{regular} if each vertex has exactly $z$ incoming arcs, $z$ outgoing arcs, and $r$ undirected edges. If, in addition, the graph has girth $g$ (i.e., the length of its shortest cycle is $g$), it is referred to as a \emph{$[z,r;g]$-mixed graph} of girth $g$.
A \emph{$[z,r;g]$-mixed cage} is a $[z,r;g]$-mixed graph that has the smallest possible order, where the \emph{order} refers to the number of vertices in the graph. The minimum order of such graphs is denoted by $n[z,r;g]$.
\\

The concept of mixed cages was first introduced by Araujo-Pardo, Hernández-Cruz, and Montellano-Ballesteros~\cite{araujo2019mixed}. In their work, they constructed specific examples of mixed cages for particular parameter values and derived several lower bounds. They established a lower bound on the order of $[1,r;g]$-mixed cages, which is now known as the AHM bound. This bound is derived from a version of Moore’s lower bound (see e.g.~\cite{exoo2012dynamic}).
 
The authors in~\cite{araujo2019mixed} also showed that the AHM bound can be achieved when $r=2$, regardless of the value of $g$. Additionally, Exoo~\cite{exoo2023mixed} demonstrated that this bound is attainable for $r=3$ and $g=5,6$.

Araujo-Pardo, {De la Cruz}, and González-Moreno~\cite{araujo2022monotonicity} showed that the order of a $[z,r;g]$-mixed cages is monotone with respect to $g$ when $z\in \{ 1,2 \}$. Furthermore, it was established that $[z,r;g]$-mixed cages are 2-connected, and, for arbitrary values of $z$, the corresponding mixed graphs are strongly connected. In the specific case for $g = 5$ the study provided additional constructions alongside improved lower bounds. These results are obtained through the use of incidence geometry, with a particular focus on the elliptic semiplane of type $C$--a structure also referred to in the geometric literature as the biaffine plane of type 1 (see, e.g., \cite{Funk2009,araujo2012familiesofsmall,araujo2022monotonicity} for the former terminology and \cite{kiss2019finite,araujo2024mixedcagesgirth6,APKP2025} for the latter).

Exoo~\cite{exoo2023mixed} concentrated on $z \in \{1, 2\}$ and showcased a range of mixed cages along with upper bounds for various values of $r \leq 5$ and $g \leq 8$, employing computer searches.

Jajcayov\'a and Jajcay~\cite{jajcajova2024totallyregular} developed $[z,r;g]$-mixed graphs by substituting some edges of certain regular undirected graphs with arcs.

Araujo-Pardo and Mendoza-Cadena~\cite{araujo2024mixedcagesgirth6} constructed a family of mixed graphs of girth6 based on the incidence graph of an elliptic semiplane of type $C$ (also known as a biaffine plane of type 1). This semiplane arises by deleting a \emph{flag} from a projective plane--that is, a point $p$ and an incident line~$l$, along with all lines through~$p$ and all points on~$l$.
In a subsequent work~\cite{araujopardo2025noteGirthFive}, they constructed a family of mixed graphs of girth 5 using the incidence graph of an elliptic semiplane of type $L$ (also known as a biaffine plane of type 2). In this case, the construction involves deleting an \emph{antiflag} from a projective plane--again, a point $p$ and a line $l$, but this time with $p$ and $l$ non-incident-along with all lines through~$p$ and all points on~$l$. This approach was inspired by a similar construction by Abajo and C. Balbuena, Bendala and Marcote~\cite{abajo2019improving} for obtaining small regular graphs of girth~5.

In this manuscript, we present a new construction of mixed graphs with a girth of 6, utilizing the incidence graphs of elliptic semiplanes of type L. Our approach is similar to that in \cite{araujo2024mixedcagesgirth6}. As in our previous work, we employ a projective plane over a finite field. However, in this case, we utilize the incidence graph of an elliptic semiplane of type $L$ along with the multiplicative group of the finite field, as discussed in \cite{araujopardo2025noteGirthFive}. In contrast, our earlier construction relied on the additive group and the incidence graph of an elliptic semiplane of type $C$ (or a biaffine plane of type 1). It is important to note that while \cite{araujo2024mixedcagesgirth6} considered only prime numbers, in this manuscript we allow $q$ to be a prime power.

The importance of these constructions lies in their ability to offer broad, general frameworks on the construction of mixed graphs of given girth. In addition, the constructions given in \cite{araujo2024mixedcagesgirth6} and those given in this paper offer the first established upper bounds of mixed cages of girth $6$.

To compare the values between the lower bounds for the mixed cages of girth $6$ and the upper bounds presented in ~\cite{araujo2024mixedcagesgirth6} and this paper, we provide Table \ref{tab:results} at the end of this paper.
The lower bounds consist of the AHM bound when $z=1$; otherwise, we use the bound derived in~\cite{araujo2024mixedcagesgirth6}, which extends the ideas from~\cite{araujo2019mixed} used to obtain the AHM bound. This second bound assumes that the $[z,0;g]$-mixed graph (which is, in fact, a directed graph) constructed in~\cite{behzad1970minimal} is a directed cage--that is, a directed graph of minimum order with in-degree and out-degree equal to~$z$ and directed girth~$g$.
Values marked with a $\star$ in the table indicate that the bound relies on this conjecture.

\section{Preliminaries \label{sec:preliminaries}}
To be self-contained, we describe in detail three graphs that are obtained from projective planes of order $q$. For an introduction, see e.g. \cite{balbuena2008incidence, van2001course, kiss2019finite, godsil2013algebraic}.

\paragraph{Incidence graph of the projective plane on the Galois field.}  Let $q$ be a prime power, and let $\mathbb{F}_q$ be the Galois field of order $q$. 
We describe the projective plane $PG(2; q)$.
Let $L_\infty$ and $P_\infty$ be the incident line and point in the projective plane, called the infinity line and infinity point, respectively.
The classes of lines (points)  are denoted by $L_i$ ($P_i$)  for $i \in \mathbb{F}_q$.
Each line $L$ has $q+1$ points and each point $P$ is incident to $q+1$ lines. 
Any line in $L_m$ described by $y = m x +b$  is denoted by $[m,b]$, and any point in $P_x$ is denoted by $(x,y)$. 
For each $i \in  \mathbb{F}_q$, the set of points incident to $L_i$ is $\{ (i,j) \mid j \in  \mathbb{F}_q\}$ and the set of lines incident to $P_i$ is $\{ [i,j] \mid j \in   \mathbb{F}_q \}$.
Finally, for $m,b \in  \mathbb{F}_q$ the line $[m,b]$ is incident to all points $(x,y)$ such that $y = m x +b$ holds for $x,y \in  \mathbb{F}_q$.

The projective plane $PG(2; q)$ has an incidence graph associated that we denote by $G_{(2,q)}$. Each line and point has a unique associated vertex $v \in V(G_{(2,q)}) $. 
Similarly, $uv \in E(G_{(2,q)})$ if and only if  their associated line and point are incident in  $PG(2; q)$. For an easy reading, we may refer to ``line $[m,b]$'' (point $(x,y)$) instead of ``vertex associated to the line $[m,b]$'' (point $(x,y)$). Note that $G_{(2,q)}$ has order $2q^2 + 2q + 2$, diameter 3 and girth 6. In fact, it is known that this incidence graph is a $[0,q+1;6]$-mixed cage (undirected cage) that attains the Moore bound.

Let us describe the incidence graph of an elliptic semiplane of type $L$. This graph is derived from the projective plane. We select a line and remove it along with all of its incident points. Next, we choose another point from the remaining points and delete it, along with all its incident lines. In particular, picking line $L_{0}$ and point $P_{0}$, we obtain the following graph.

\begin{defn}[Graph $G_q$]\label{def:graph_G_q}
    Let $\mathbb{F}_q^* =  \mathbb{F}_q \setminus \{ 0 \}$. 
    For a fixed $b,y \in \mathbb{F}_q$, let us denote the sets of vertices $\mathcal{L}_b \coloneq \{ [m,b] \mid m \in \mathbb{F}^*_q\}$,  $\mathcal{P}_y \coloneq \{ (x,y) \mid x \in \mathbb{F}^*_q \}$,  $\mathcal{L}_\infty \coloneq \{ L_i \mid i \in \mathbb{F}^*_q \}$, and $\mathcal{P}_\infty \coloneq \{ P_i \mid i\in \mathbb{F}^*_q\}$.
    Finally, we define $G_q \coloneq G_{(2,q)}[ \{ \mathcal{L}_b \mid b \in  \mathbb{F}_q\} \cup \{\mathcal{P}_y  \mid y \in \mathbb{F}_q \} \cup \{ \mathcal{L}_\infty , \mathcal{P}_\infty \}]$.
\end{defn}

This graph has $2(q-1)(q+1)$ nodes and girth 6.

\paragraph{Primitive element.}  A \emph{primitive element} $\xi \in \mathbb{F}_q^*$ is a multiplicative generator of the field $\mathbb{F}_q$, such that the non-zero elements of  $\mathbb{F}_q$ can be written as $\xi^i$ for some $i\in \{ 0, \dots, q-2 \} $, that is,  $ \mathbb{F}_q = \{  0, \xi^0, \xi^1, \xi^2, \dots, \xi^{q-2}\}$. This is due to the isomorphism between the multiplicative group $\mathbb{F}^*_q$ and the additive group $\mathbb{Z}_{q-1}$, as described by Abajo, Balbuena, Bendala, and Marcote~\cite{abajo2019improving}. For example, a primitive element of the Galois Field of order 8 is $\alpha$ ($\alpha$ a root of the irreducible polynomial $f(x)=x^3+x+1$ over $\mathcal{Z}_2$), and a primitive element of $\mathbb{F}_7$ is 2. See e.g.~\cite{jones1998elementary}.
 
\paragraph{The circulant digraph $\overset{\bm{\rightarrow}}{\bm{C}}_{\bm{q}} \bm{{(i_1, \dots, i_k)}}$.} 
Consider the field $\mathbb{Z}_q$ and $i_1, \dots, i_k \in \mathbb{Z}_q$. A \emph{circulant digraph} $\overset{\rightarrow}{C}_q(i_1, \dots, i_k)$ has a vertex $v_i$ for each $i \in \{ 0, 1, \dots, q -1 \}$. There exists an arc $(v_{a},v_{b})$ if and only if $b \equiv  a + i \mod q$ for some $i \in i_1, \dots, i_k$. It is known that if $q = z(g-1) + 1$, then $\overset{\rightarrow}{C}_q( 1, 2, \dots, z)$ is a $[z,0;g]$-mixed graph (see e.g.~\cite{behzad1970minimal, araujo2009dicages}). 

\section{A family of mixed graphs with girth 6\label{sec:new_family_girth_6}}
For the Galois Field $ \mathbb{F}_q $ where $ q \geq 7 $ is a prime power, we consider the multiplicative group. Let $ \xi \in \mathbb{F}^*_q $ be a primitive element. We examine the graph $ G_q $, which has a girth of 6 and contains $ 2q^2 - 2 $ nodes. We describe a mixed graph of girth 6, where the nodes are arranged according to their second coordinate. Each node has a corresponding copy, and we add a circulant between the vertices and their copies.

Let $ q $ be a prime power. We consider the bipartite graph $ G_q $ over $ \mathbb{Z}_q $, as described in Section~\ref{sec:preliminaries}. We take a copy of $ G_q $, denoted as $ G'_q $. For clarity, whenever $ v' \in V(G'_q) $ is referenced, we assume that $ v' = v $ — that is, $ v' $ represents the copy vertex of the original vertex $ v \in V(G_q) $; the same applies to the edges. The notations $ \mathcal{L}'_b, \mathcal{P}'_y, \mathcal{L}'_\infty, \mathcal{P}'_\infty $ for the copy elements follow the conventions established in Definition~\ref{def:graph_G_q}.

We will introduce a value $ p $, which refers to the length of the longest ``jump'' made in a circulant. The value of $ p $ depends on the parity of $ q $. The number of arcs within the circulants also depends on the parity of $ p $, which relates to the integrality required of these two numbers. Overall, our constructions exhibit slight variations in the number of arcs based on the parities of $ q $ and $ p $. We have three cases that primarily depend on the parity of $ q $.

\subsection{Construction when $q$ is even}
\label{sec:q_even}
Consider $q$ even. Let $p \in \mathbb{Z}$ be such that $2(q-1) = 4p +2$, or equivalently, $p = \frac{q-2}{2}$ which is an integer number because $q$ is an even prime power, or simply a power of $2$. It is important to note that $p$ is always odd. To see this, given that $q$ is a power of $2$, expressed as $2^\alpha$ for some positive integer $\alpha$ then $p = \frac{2(2^{\alpha - 1}-1)}{2} =2^{\alpha - 1}-1$.

In the following, all the computations are made modulo $q-1$ because calculations are performed in the first coordinate, where there are $q-1$ elements from the set $\mathbb{F}^* = \{ \xi^0, \xi^1, \xi^2, \dots, \xi^{q-2}\}$.

For a fixed $y \in \mathbb{F}_q$, we add circulants to the sets $\mathcal{P}_y \cup \mathcal{P}'_y$ and to $\mathcal{L}_\infty \cup \mathcal{L}'_\infty$. Let us explicitly describe the set of arcs to be added.

Let $i \in \{ 1,2, \dots, \frac{p+1}{2} \}$ --note that $\frac{p+1}{2}$ is an integer number as $p$ is an odd number.
For a jump of length $i$ and for each $x \in \mathbb{F}^*_q$ we add arcs as follows: from a point to its copy, we include the arc $\big( (x,y),({(x\xi^{i-1})'},y') \big)$  and from a copy to the original we add arc $\big( (x',y'),(x\xi^i,y) \big)$. For the infinity lines we add arcs 
$\big( L_x, L_{x\xi^{i-1}}' \big)$ and $\big( L_x', L_{x\xi^i} \big)$. 
Similarly for lines and the infinity points. For fixed $b \in \mathbb{F}_q$ and for jump $i$ we add arcs $\big( [m,b],[(m/\xi^{i})',b'] \big)$, $\big( [m',b'],[m/\xi^{i-1},b] \big)$ and $\big( P_m, P'_{m/\xi^{i}} \big)$, $\big( P_m', P_{m/\xi^{i-1}} \big)$ for each $m \in \mathbb{F}^*_q$. 

Note that in all cases, we add a circulant $\overset{\rightarrow}{C}_{q-1}(1,3,5,\dots, p)$ where the vertices $v_1, v_2 \dots, v_{2q -2}$ are arranged in increasing order based on their exponents, with the original vertices given priority.

Let us call this resulting mixed graph $H_{q,p}$. Figure~\ref{fig:example} shows an example of this constructions for $q=8$ and $p = 3$.
 
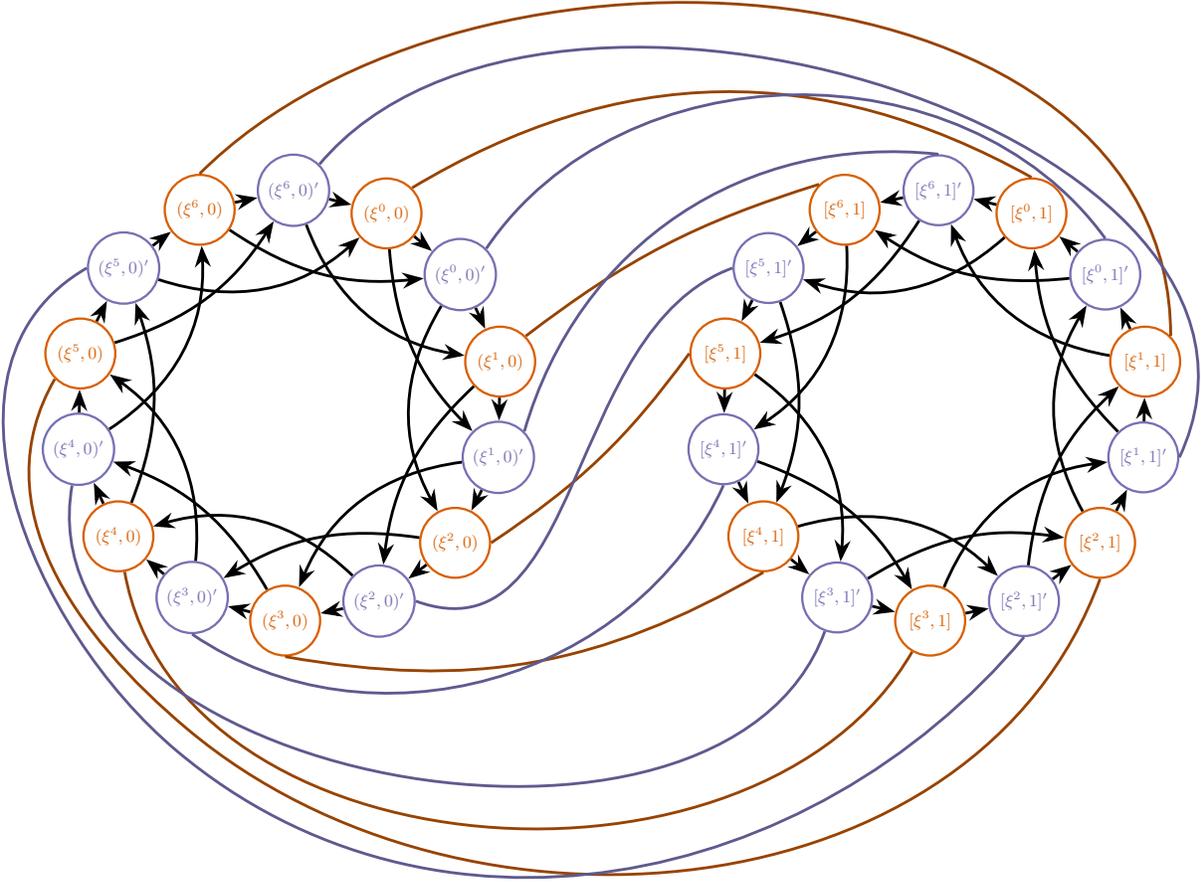
\begin{figure}[ht!]
    \centering
    \resizebox{\textwidth}{!}{
    \trimbox{40pt 40pt 60pt 40pt}{
    \begin{tikzpicture}[node distance=7em, 
        myArc/.style={draw,line width = 1.5pt, -{Stealth[length=4mm]}},
        myEdge/.style={line width = 1.5pt}, 
        stateColor/.style={circle,  minimum size=3.7em, draw, line width = 1.2pt,MyOrange},
        stateColorCopy/.style={circle,  minimum size=3.7em, draw, line width = 1.2pt,MyPurple}]   
       \node[minimum size={2*4cm},regular polygon,regular polygon
        sides=14,rotate=11.75] at (0,0) (14-gon) {};
        \foreach \p [count=\n from 0] in {7,...,1}{%
            \tikzmath{integer \x, \y;
            \x = 2*\p; \y = 2*\p-1;}
            \node[stateColorCopy] at (14-gon.corner \y) (copy\n) {\small$(\xi^{\n},0)'$};
            \node[stateColor] at (14-gon.corner \x) (point\n) {\small$(\xi^{\n},0)$};};
        \foreach \i in {0, ..., 6}{
            \tikzmath{integer \j1, \j2;
            \j1 = Mod(\i+1,7); \j2 = Mod(\i+2,7); }
            \draw[myArc] (point\i) to (copy\i);
            \draw[myArc] (copy\i) to (point\j1);
            \draw[myArc, bend right=20] (point\i) to (copy\j1);
            \draw[myArc, bend right=30] (copy\i) to (point\j2);
            }
        \node[minimum size={2*4cm}, regular polygon,regular polygon
        sides=14,rotate=11.75] at ($(14-gon) + (12,0)$) (14-gonLINE) {};
        \foreach \p [count=\n from 0] in {7,...,1}{%
            \tikzmath{integer \x, \y;
            \x = 2*\p; \y = 2*\p-1;}
            \node[stateColorCopy] at (14-gonLINE.corner \y) (LineCopy\n) {\small$[\xi^{\n},1]'$};
        \node[stateColor] at (14-gonLINE.corner \x) (line\n) {\small$[\xi^{\n},1]$};};
    
        \foreach \i in {0, ..., 6}{
            \tikzmath{integer \j1, \j2;
            \j1 = Mod(\i-1,7); \j2 = Mod(\i-2,7); }
            \draw[myArc] (line\i) to (LineCopy\j1);
            \draw[myArc] (LineCopy\i) to (line\i);
            \draw[myArc, bend left=30] (line\i) to (LineCopy\j2);
            \draw[myArc, bend left=20] (LineCopy\i) to (line\j1);
            }
         \foreach \i \j \bending in {0.north east/0.north/30, 
         3.south/4.south/-20,2.east/5.west/-10,1.north east/6.north west/10}{
            \draw[myEdge, bend left=\bending, MyOrange!70!black] (point\i) to (line\j);
            }
        \draw[myEdge,MyOrange!70!black] (point6.north) to  [out=45, in=90] (line1.north east);
        \draw[myEdge,MyOrange!70!black] (point4) to  [out=280, in=240] (line3);
        \draw[myEdge,MyOrange!70!black] (point5.south west) to  [out=240, in=140] ($(point5) + (2,-7)$) to  [out=140-180, in=250] (line2.south); 
          \foreach \i \j \bending in {0.north east/0.north/55, 
          3.south/4.south/-50,
          1.north east/6.north/40}{
            \draw[myEdge, bend left=\bending, MyPurple!80!black] (copy\i) to (LineCopy\j); 
            }
        \draw[myEdge,MyPurple!80!black] (copy6.north east) to  [out=50, in=65] (LineCopy1.east);
        \draw[myEdge,MyPurple!80!black] (copy5.west) to  [out=210, in=140] ($(copy5) + (1,-9)$) to  [out=140-180, in=230] (LineCopy2.south); 
        \draw[myEdge,MyPurple!80!black] (copy4) to  [out=260, in=250] (LineCopy3); 
        \draw[myEdge,MyPurple!80!black] (copy2.east) to  [out=-20, in=200] (LineCopy5.west); 
    \end{tikzpicture}
    }
    }
\caption{Set of vertices $\mathcal{P}_0$ and $\mathcal{L}_1$ of graph $H_{q,p}$ for $q=8$, $p=3$. This is the Galois field of order 8.}
\label{fig:example}
\end{figure}

Note that the circulants define a partition on the nodes as follows. Points are arranged according to their second coordinate, hence we have a partition with $2(q+1)$ parts, each part containing $2(q-1)$ elements. That is, we have part $\mathcal{P}_y\cup\mathcal{P}'_y$ for a fixed  $y \in \mathbb{F}_q$, part $\mathcal{L}_b\cup \mathcal{L}'_b$ for a fixed $b\in \mathbb{F}_q$, part $\mathcal{L}_\infty\cup \mathcal{L}'_\infty$ and part $\mathcal{P}_\infty\cup \mathcal{P}'_\infty$. Using this partition, it is not difficult to see that for each point $(x,y)$ (which belongs to part $\mathcal{P}_y$ for a fixed $y \in \mathbb{F}^*_q$), there exists a unique line $[\frac{y-b}{x},b]$ in part $\mathcal{L}_b$ for $b \in \mathbb{F}_q$; this is well-defined as $x\in \mathbb{F}^*_q$. Hence, for any fixed pair $y,b \in \mathbb{F}_q$ there is a matching between parts $\mathcal{P}_y$ and $\mathcal{L}_b$. The same holds for parts $\mathcal{P}_y$ for a fixed  $y \in \mathbb{F}_q$ and $\mathcal{L}_\infty$, and for parts $\mathcal{L}_b$ for a fixed $b\in \mathbb{F}_q$ and $\mathcal{P}_\infty$. Note that parts $\mathcal{L}_\infty$ and $\mathcal{P}_\infty$ share no edge. All of these holds also for the copies.

\begin{lem}\label{lem:directed_cycles}
    The circulants of the graph $H_{q,p}$ has directed girth 6 when $q$ is even.
\end{lem}
\begin{proof}
    Consider the circulant of part $\mathcal{P}_y \cup \mathcal{P}'_y$ for a fixed $y \in \mathbb{F}_q$. We know there exists a directed cycle of length $6$, e.g. for $x \in\mathbb{F}^*$, we obtain the following directed cycle by giving 4 jumps of length $p$, or equivalently for $i=\frac{p+1}{2}$, and then 2 jumps of length $i=1$: $\left( (x,y), ({x\xi^{\frac{p-1}{2}}} ',y'),
    (x\xi^{p},y), ({x\xi^{\frac{3p-1}{2}}}' ,y'), (x\xi^{2p},y), ({x\xi^{2p}}',y') \right)$ which is well defined as there is an arc from $({x\xi^{2p}}',y')$ to $(x,y)=(x\xi^{2p + 1},y)$ --recall that $2(q-1) = 4p + 2$. 
    
    We claim that there is no smaller directed cycle than this type. Let $K$ be a directed cycle of smallest length.
    First, note that $\mathcal{P}_y \cup \mathcal{P}'_y$ is bipartite (in the directed subgraph), and thus the length of $K$ cannot be 5 or 3. Clearly $|K|=2$ is not possible by construction. 
    We show that $|K| \neq 4$.
    Suppose we start at position $(x, y)$. For these computations, recall that the circulant is of the form $\overset{\rightarrow}{C}_{q-1}(1, 3, 5, \ldots, p)$, meaning we only allow jumps of odd lengths. Since $p$ represents the length of the longest jump, we can see that by making four jumps of length $p$, we obtain the equation $4p = 2(q-1) - 2 < 2(q-1)$. This indicates that we are two positions away from the origin $(x, y)$, even after making four jumps of the maximum length. Consequently, this establishes that $|K| = 6$ based on the cycle shown above.
    
    Similarly for the other parts.
\end{proof}

\begin{lem}\label{lem:girth}
    The graph $H_{q,p}$ has girth 6 when $q$ is even and $p$ is odd.
\end{lem}
\begin{proof}
    The graph has undirected girth equals to 6 by construction (two copies of graph $G_q$), and directed girth equals 6 by Lemma~\ref{lem:directed_cycles}.
    We verify that any mixed cycle has a length at least 6.

    Let $K$ be the smallest mixed cycle of the graph.    
    Based on the partition and the matching among parts discussed earlier, it is clear that no mixed cycles of lengths 2 and 3 exist. Also, $|K| \neq 5$ because the graph is bipartite with parts $\bigcup_b \{ \mathcal{L}_b \} \bigcup_y \{\mathcal{P}'_y\}  \bigcup\{ \mathcal{L}_\infty\}  \bigcup\{ \mathcal{P}'_\infty\}$ and $\bigcup_b \{ \mathcal{L'}_b \} \bigcup_y \{\mathcal{P}_y\}  \bigcup\{ \mathcal{L}'_\infty\}  \bigcup\{ \mathcal{P}_\infty\}$.    
    We show that $|K|\neq 4$. As we only have arcs between lines (points) that share the first same coordinate (one in vertex is in $G_q$ and the other in $G'$), in order to have a mixed cycle of length 4, we need two edges: one in $B_q$ and one in $B'_q$. Thus, we require two arcs, one leaving $B_q$ (and entering $B'_q$) and one entering it  (and leaving $B'_q$).        
    Consider the point $(x,y)$ which is edge-adjacent to line $[m,b]$. This line $[m,b]$ is arc-adjacent to a line in $G'_q$ of the form $[(m/\xi^{j})',b']$ for some $j \in \{1 , 2, \dots, \frac{p-1}{2} \}$. The line $[({m/\xi^j})',b']$ is edge-adjacent to a point $(x', ({m/\xi^j}x + b)' )$. Even if $y= ({m/\xi^j}x + b)'$, there is no arc from $(x', ({m\xi^j}x + b)' )$ to $(x,y)$ by construction. We conclude that there is no cycle of length 4. This concludes the proof.
\end{proof}

\begin{thm}
    For all even prime power $q$ and $p$ such that $2(q-1)=4p+2$, the graph $H_{q,p}$ is a $[\frac{p+1}{2},q;6]$-mixed graph.
\end{thm}
\begin{proof}
    The parameters follow by construction and by Lemma~\ref{lem:girth}.
\end{proof}

This gives us an upper bound on the order of a mixed-cage of order $6$ with specific parameters.
\begin{cor}\label{cor:order_even}
    Let  $q$ be an even prime power.
    Let $n[z,r;6]$ be the minimum order of a $n[z,r;6]$-mixed cage. Then, $n[\frac{q}{4},q;6]\leq 4q^2-4$.
\end{cor}

\subsection{Construction when $q$ is odd}
\label{sec:q_odd}
Consider $q$ odd. Let $p \in \mathbb{Z}$ be such that $p=\frac{q-1}{2}-1$, which is an integer number because $q$ is odd. 

If $p$ is odd, then the construction of the graph $H_{q,p}$ is the same.
Otherwise, when $p$ is even we add circulant $\overset{\rightarrow}{C}_{q-1}(1,3,5,\dots, p-1)$ to each of the parts.
Note that there there is no cycle of length 4. If we consider four jumps of length $p-1$ from an starting vertex $v$, then $4(\frac{q-1}{2}-2) = 2(q-1) - 8 < 2(q-1)$, implying that we are eight steps away from $v$. As we give only odd jumps, we need at least two more jumps to reach $v$. Hence, $p$ being the longest jump is well defined.

The rest of the proofs follow in the same way as Section~\ref{sec:q_even}.

\begin{cor}
    Let $q$ be an odd prime power and let $p=\frac{q-1}{2}-1$. 
    If $p$ is odd, then the graph $H_{q,p}$ is a $[\frac{p+1}{2},q;6]$-mixed graph.
    If $p$ is even, then the graph $H_{q,p}$ is a $[\frac{p}{2},q;6]$-mixed graph.
\end{cor}

\begin{cor}
    Let  $q$ be an odd prime power.
    Let $n[z,r;6]$ be the minimum order of a $n[z,r;6]$-mixed cage. 
    Then, if $\frac{q-3}{2}$ is odd then
    $n[\frac{q-1}{4},q;6]\leq 4q^2-4$, otherwise $n[\frac{q-3}{4},q;6]\leq 4q^2-4$
\end{cor}

\section*{Final discussion}
As mentioned in the introduction, in this paper we allow $q$ to be any prime power, whereas the construction in~\cite{araujo2024mixedcagesgirth6} considers only prime values of~$q$. It is important to highlight that the two constructions may differ in the number of arcs. For instance, when $q = 7$, the construction in~\cite{araujo2024mixedcagesgirth6} yields a mixed graph with 2 arcs, while the construction presented here produces a mixed graph with only 1 arc. Even when both constructions yield the same parameters, the one described here might differ slightly in terms of order. For example, when $q = 13$, our construction generates a mixed graph with 672 vertices, while the construction in \cite{araujo2024mixedcagesgirth6} results in 676 vertices. These differences are also reflected in Table~\ref{tab:results}.

Table~\ref{tab:results} summarizes the current best-known lower and upper bounds for $[z,r;6]$-mixed graphs, for $0 \leq z \leq 5$ and $2 \leq r \leq 19$. The lower bounds include the AHM bound for $z = 1$, and the conjecture-dependent bound from~\cite{araujo2024mixedcagesgirth6}, as previously discussed. Values that depend on this conjecture are marked with the symbol $\star$. 

\begin{table}[h!]     
    \centering
    \arrayrulewidth 1pt
    \caption{Best known bounds on the order of mixed cages with girth 6. Entries marked with a $\star$ are conditional on a conjecture.}
    \label{tab:results}
    \begin{tabular}{|C{3em}C{3em}C{3em}|C{7em}|C{5em}|>{\arraybackslash}m{15em}|}
    \hline
    \textbf{Arcs} & \textbf{Edges} & \textbf{Girth} & \textbf{Lower bound} & \textbf{Exact} & \textbf{Upper bound}\\ \hline
     \arrayrulecolor{gray!30}
    0 & 2 & 6 &  &  & 12 (Prev. 16~\cite{araujo2024mixedcagesgirth6})\\\hdashline[1pt/1pt]	
    0 & 3 & 6 &  & &32\\\hdashline[1pt/1pt]	    
    1 & 2 & 6 & 18\phantom{$^\star$} & 18~\cite{araujo2019mixed} &\\\hdashline[1pt/1pt]	 
    1 & 3 & 6 & 30\phantom{$^\star$} & 30~\cite{exoo2023mixed} & 36 \cite{araujo2024mixedcagesgirth6}\\\hdashline[1pt/1pt]	
    1 & 4 & 6 & 46\phantom{$^\star$} & & 48~\cite{exoo2023mixed} (60 by Cor.~\ref{cor:order_even})\\\hdashline[1pt/1pt]	
    1 & 5 & 6 & 66\phantom{$^\star$} & & 72~\cite{exoo2023mixed} (96 by Cor.~\ref{cor:order_even}, Prev. 100~\cite{araujo2024mixedcagesgirth6})\\\hdashline[1pt/1pt]	
    1 & 6 & 6 & 90\phantom{$^\star$} & 90~\cite{exoo2023amixedgraphachieving} & \\\hdashline[1pt/1pt]	    
    1 & 7 & 6 & 118\phantom{$^\star$} & &192\\\hdashline[1pt/1pt]	
    2 & 2 & 6 & 18\phantom{$^\star$} & 27~\cite{exoo2023mixed} &\\\hdashline[1pt/1pt]	
    2 & 7 & 6 & $123^\star$ & &196\cite{araujo2024mixedcagesgirth6}\\\hdashline[1pt/1pt]	
    2 & 8 & 6 & $155^\star$ & &252 \\\hdashline[1pt/1pt]	
    2 & 9 & 6 & $191^\star$ & &320 \\\hdashline[1pt/1pt]	
    2 & 11 & 6 & $275^\star$ & &480 \\\hdashline[1pt/1pt]	
    3 & 11 & 6 & $280^\star$ & &484 \cite{araujo2024mixedcagesgirth6}\\\hdashline[1pt/1pt]	
    3 & 13 & 6 & $380^\star$ & &672 (Prev. 676 \cite{araujo2024mixedcagesgirth6})\\\hdashline[1pt/1pt]	
    4 & 16 & 6 & $565^\star$ & &1020 \\\hdashline[1pt/1pt]	
    4 & 17 & 6 & $633^\star$ & &1152 (Prev. 1156\cite{araujo2024mixedcagesgirth6})\\\hdashline[1pt/1pt]	
    4 & 19 & 6 & $781^\star$ & &1440\\\hdashline[1pt/1pt] 
    5 & 19 & 6 & $786^\star$ & &1444\\\hdashline[1pt/1pt]	
    \arrayrulecolor{black}
    \hline
    \end{tabular}
\end{table}

\subsubsection*{Acknowledgments}

G. Araujo-Pardo was supported by PAPIIT-UNAM-M{\' e}xico IN113324.

G. Araujo-Pardo and L.M. Mendoza-Cadena were supported CONAHCyT: CBF2023-2020-552 M{\' e}xico.

L. M. Mendoza-Cadena was supported by the Ministry of Innovation and Technology of Hungary from the National Research, Development and Innovation Fund -- grant ELTE TKP 2021-NKTA-62.


\bibliographystyle{abbrv}
\bibliography{improving_girth6.bib}
\end{document}